\documentclass[12pt,leqno,amscd,amssymb,verbatim, url]{amsart}
\usepackage{amsfonts,amssymb,graphicx}
\usepackage{amsmath,amscd}
\usepackage{hyperref}
\usepackage{mathrsfs}
\let\oldsection=

\renewcommand{\subsection}[1]{\par\vspace{.18in}\noindent\addtocounter{subsection}{1}\setcounter{equation}{0}{\bf\thesubsection\hspace{9pt}#1}}

\usepackage[english]{babel}
\usepackage[latin1]{inputenc}
\usepackage{times}
\usepackage[T1]{fontenc}
\usepackage{amsthm, amsmath, amssymb}
\usepackage{tikz}
 
\usetikzlibrary{arrows,shapes}
\usetikzlibrary{decorations.pathmorphing}
\usetikzlibrary{calc}

\newtheorem{thm}{Theorem}[subsection]

\makeatletter\let\c@fact\c@theorem\makeatother\newtheorem{lem}[thm]{Lemma}
\newtheorem{cor}[thm]{Corollary}

\newtheorem{prop}[thm]{Proposition}

\theoremstyle{definition}

\newtheorem{ques}[thm]{Question}
\newtheorem{rem}[thm]{Remark}

\numberwithin{equation}{subsection}
\numberwithin{thm}{section}

\newcommand{\grExt}{\text{\rm ext}}

\newcommand{\wfa}{\widetilde{\mathfrak a}}

\newcommand{\wgr}{\widetilde{\text{\rm gr}}\,}

\newcommand{\ssC}{{\mathcal C}}

\newcommand{\wrad}{\widetilde{\text{\rm rad}}\,}

\newcommand{\gr}{\text{\rm gr}}

\newcommand{\sB}{{\mathcal {B}}}
\newcommand{\sC}{{\mathscr {C}}}

\newcommand{\Ext}{{\text{\rm Ext}}}

\newcommand{\htt}{\operatorname{ht}}
\newcommand{\op}{{\text{\rm op}}}

\newcommand{\Hom}{\text{\rm Hom}}

\newcommand{\End}{\operatorname{End}}

\newcommand{\ind}{\operatorname{ind}}
\newcommand{\sE}{\operatorname{{\mathscr E}}}

\newcommand{\sO}{{\mathscr{O}}}

\newcommand{\rad}{\operatorname{rad}}

\newcommand{\Dist}{\operatorname{Dist}}

\newcommand{\rDelta}{\Delta^{\text{\rm red}}}
\newcommand{\rnabla}{\nabla_{\text{\rm red}}}

\newcommand{\wA}{{\widetilde{A}}}

\newcommand{\wU}{{\widetilde U}}

\newcommand{\Jan}{\Gamma_{\text{\rm Jan}}}

\newcommand{\wL}{{\widetilde{L}}}

\newcommand{\blist}{\begin{list}{\rom{(\roman{enumi})}}{\setlength
{\leftmarg in}{0em} \setlength{\itemindent}{7ex}
\setlength{\labelsep}{2ex}\setlength{\listparindent}{\parindent}
\usecounter{enumi}}}
\newcommand{\elist}{\end{list}}

\begin{document} 
\begin{abstract} Let $G$ be a simple, simply connected algebraic group over an algebraically closed field
of positive characteristic $p$. In recent work, \cite{PS10}, \cite{PS11} and \cite{PS13}   the authors have studied a graded analogue of the category of rational $G$-modules. These gradings are not natural but are
``forced" on
 related algebras though filtrations, often obtained from appropriate quantum structures. This paper presents new
results on Koszul modules for the graded algebras obtained through this forced grading process. Most of
these results require that the Lusztig character formula holds for all restricted $p$-regular weights, but the paper begins to investigate
how these and previous results might be established when the Lusztig character formula is only assumed to hold on a
proper poset ideal in the
Jantzen region. This opens up the possibility of inductive arguments.
    \end{abstract}

 \title[Some Koszul properties of standard and irreducible modules]{Some Koszul properties of standard and
 irreducible modules}\author{Brian J. Parshall}
\address{Department of Mathematics \\
University of Virginia\\
Charlottesville, VA 22903} \email{bjp8w@virginia.edu {\text{\rm
(Parshall)}}}
\author{Leonard L. Scott}
\address{Department of Mathematics \\
University of Virginia\\
Charlottesville, VA 22903} \email{lls2l@virginia.edu {\text{\rm
(Scott)}}}

\thanks{Research supported in part by the National Science
Foundation}
\maketitle

\section{Introduction} Let $G$ be a simple, simply connected algebraic group over an algebraically closed
field $k$ of characteristic $p>0$.
In a series of recent papers \cite{PS10}, \cite{PS11},  and \cite{PS13}, the authors have obtained results on the modular representation theory of $G$ using new, ``forced grading," methods. This work generally assumed that $p\geq 2h-2$, though \cite{PS12}, which also employed forced
grading techniques, suggests a method for removing that condition. In addition, it was required that the Lusztig character
formula (LCF) hold for all restricted weights. Among the new results obtained by these methods was a verification of a conjecture of Jantzen \cite{Jan} on $p$-Weyl filtrations of Weyl modules, assuming that
the LCF holds (see \cite{PS11}), and, in addition. many new results
on the $\nabla$-filtrations (i.~e., good filtrations) of the $\Ext$-groups (after untwisting) for the first Frobenius kernel $G_1$ of $G$ (see \cite{PS13}). 

  The representation theory of $G$ can be studied by means of quasi-hereditary algebras $A_\Gamma$ attached to a finite
set $\Gamma$ of dominant weights which is a poset ideal in the dominance ordering or the Bruhat-Chevalley ordering. (See \S2.1 below.)  In turn, the algebra $A_\Gamma$ has a (forced) graded version, which we
denote $B:=\wgr A_\Gamma$, which is, remarkably, also quasi-hereditary. In addition, when the LCF holds
for all restricted weights, the algebra $B$ has been shown to be a Q-Koszul algebra in the sense
of \cite[Defn. 3.6]{PS13}. Koszul algebras are themselves examples of Q-Koszul algebras, and, when
$\Gamma$ is contained inside the Jantzen region  Let $G$ be a simple, simply connected algebraic group over an algebraically closed field
of positive characteristic $p$. In recent work,    the authors have studied a graded analogue of the category of rational $G$-modules. These gradings are not natural but are
``forced" on
 related algebras though filtrations, often obtained from appropriate quantum structures. This paper presents new
results on Koszul modules for the graded algebras obtained through this forced grading process. Most of
these results require that the Lusztig character formula holds for all restricted $p$-regular weights, but the paper begins to investigate
how these and previous results might be established when the Lusztig character formula is only assumed to hold on a
proper poset ideal in the
Jantzen region $\Jan:=\{\lambda\in X(T)_+\,|\,(\lambda+\rho,\alpha^\vee_0)\leq p(p-h+2)\}$. This opens up the possibility of inductive arguments., $B$ turns out to be Koszul itself. But when $\Gamma$ moves outside the Jantzen region, Koszulity generally fails. However, $B$ remains Q-Koszul, and the notion of Q-Koszulity nicely captures some of
the homological algebra of $G$, even for modules parameterized by dominant weights far from 0.  

In the representation
theory of Koszul algebras, there is an important notion of a linear (or ``Koszul") graded module, one which
has a particularly nice minimal projective resolution. This concept readily extends to the case of Q-Koszul
algebras, and the results of \S4 provide new examples of these ``Q-linear" modules, including, Q-analogues of maximal submodules of standard modules, or of any Q-linear module. In the Koszul case, we prove maximal
submodules of linear modules often have stronger resolution properties, depending on the strong linearity
of the original module.

One general aim of the project, of which this paper is a part, has been to keep conclusions, for a given algebra $A_\Gamma$
or $\wgr A_\Gamma$ or its representation theory, expressed solely in terms of $\Gamma$. This is
desirable not only for aesthetic reasons, but for applications of the theory in {\it inductive} arguments. The papers
listed above all fit this framework with the exception of \cite{PS13}, where hypotheses ``external to
$\Gamma$" were required.
For instance, it was necessary throughout most of paper \cite{PS13} to deal with primes
$p$ sufficiently large so that LCF held for all irreducible modules $L(\gamma)$ with $\gamma$ a $p$-regular restricted weight. 

Accordingly, we consider in this paper the possibility of eliminating such external hypotheses,
and we make some progress in this direction. For example, for the poset ideals $\Gamma$
considered\footnote{ We often consider only ``stable" posets
$\Gamma$, those such that whenever $\gamma=\gamma_0 + p\gamma_1\in\Gamma$, with $\gamma_0$
restricted and $\gamma_1$ is dominant, it also true that $\gamma_0\in\Gamma$. This is not a strong
requirement.} in the key Theorem \ref{maintheorem}, when the LCF is assumed, it is only for characters of irreducible modules $L(\gamma)$, when
$\gamma\in\Gamma$ is a $p$-restricted dominant weight.  Nevertheless, this theorem shows that the irreducible
modules $L(\gamma)$, $\gamma\in\Gamma$, behave homologically as if the (regular part of the) restricted enveloping algebra
were Koszul (a known consequence of the LCF \cite{AJS}). We apply Theorem \ref{maintheorem} in \S6. Here we achieve
similar ``relative to $\Gamma$" versions of the other results of \S\S3,4, though at the cost of assuming that
$p$ is fairly large (but not huge). See the discussion at the end of \S6. 

In addition, though we continue here to use $p$-regular weights, the approach of this paper, together with the others above, 
could be applied in the singular case, once some basic questions have been answered. One of these
simply asks for a small $p$ quantum version (in characteristic 0, at a primitive $p$th root of 1) of Riche's Koszulity theorem \cite{Riche}, which applies for
all weights. 
This is explicitly stated in \S5. 

Finally, in the course of proving Theorem \ref{maintheorem},  it was necessary to sharpen the deformation result
in \cite[Thm. 8.1]{PS11} which provided an integral form for the Koszul grading on the $p$-regular part of the
small quantum enveloping algebra. See Theorem \ref{thm2.2} which shows that any grade in the form is a
direct sum of its intersections with weight spaces. This result is interesting in its own right. The proof uses a corresponding result for the small quantum group itself in \cite[\S18.21]{AJS}, and we elaborate
on part of its proof in \S8 (Appendix).

\section{Preliminaries}  Unless otherwise noted, algebras over a field are assumed to be finite
dimensional.  Likewise, modules are generally taken to be finite dimensional. Let $k$ be an algebraically
closed field, generally of positive characteristic $p$. In dealing with quantum
enveloping algebras at a primitive $p$th root $\zeta$ of unity, we will need a fixed $p$-modular system $(K,\sO,k)$. Thus, $\sO$ will be a DVR with
maximal ideal ${\mathfrak m}= (\pi)$, fraction field $K$ and residue field $k=\sO/{\mathfrak m}$. It will
be assumed that $\pi=\zeta-1$. For more details, see \cite[\S2]{PS11}.

\subsection{Some standard notation.} Let $G$ be a fixed simple, simply connected algebraic group defined
over an algebraically closed field $k$ of positive characteristic. We generally (but not always) follow the notation listed in \cite[pp. 569--572]{JanB}.

Let $T$ be a fixed maximal split torus and let $R$ be the root system of $G$. Fix a set $R^+$ of positive roots corresponding to a Borel subgroup $B^+\supset T$, and let $B=B^-\supset T$ be the opposite Borel
subgroup. Regard the weight lattice $X(T)$ of $T$ as a poset by putting $\mu\leq\lambda$ if and only if
$\lambda-\mu$ is a sum of positive roots. By restricting $\leq$ to any subset $\Xi$ of weights, $\Xi$
is also a poset. A stronger partial order on $X_{\text{\rm reg}}(T)_+$ is sometimes useful (as in the
proof of Theorem \ref{maintheorem}). Namely, given a $p$-regular dominant weight, write $\gamma=w\cdot\gamma_0$
for a unique $w\in W_p$ (the affine Weyl group) and $\gamma_0\in C^+$ (the bottom dominant $p$-alcove). In this way, the intersection of any orbit with $X_{\text{\rm reg}}(T)_+$  identifies with a subset of $W_p$, and we partially order $X_{\text{\rm reg}}(T)_+$ by using
the Bruhat-Chevalley partial order on $W_p$.  

Let $\sC=G$-mod be the category of finite dimensional rational $G$-modules. If $\gamma\in X(T)_+$, then
$L(\gamma)$ (resp., $\Delta(\gamma)$, $\nabla(\gamma)$) denotes the irreducible (resp., Weyl module, dual
Weyl module) of highest weight $\gamma$. If $\Gamma$
is a set of dominant weights, let $\sC[\Gamma]$ be the full subcategory of $\sC$ generated by the irreducible
modules $L(\gamma)$ of highest weight $\gamma\in\Gamma$. If $\Gamma$ is a finite poset ideal in $X(T)_+$ or in $X_{\text{\rm reg}}(T)_+$, then $\sC[\Gamma]$ is a highest weight category with weight poset $\Gamma$. Here $\Gamma$ can
be taken to be a poset using the dominant partial order $\leq$, or any stronger order (e.~g., the Bruhat-Chevalley order, or the $\preceq$ order below.)

Let $\Gamma$ a finite ($\not=\emptyset$) poset ideal in the set $X_{\text{\rm reg}}(T)_+$ of $p$-regular dominant weights,
 satisfying the additional property that,
 if $\gamma=\gamma_0+p\gamma_1\in \Gamma$ with $\gamma_0\in X_1(T)$ (the $p$-restricted dominant weights) and $\gamma_1\in X(T)_+$, then $\gamma_0\in \Gamma$. In this case, $\Gamma$ is called {\it stable}.
For example, write $\lambda\preceq\mu$ if and only if $\mu-\lambda\in {\mathbb Q}^+
R^+$.  Then if $\Gamma$ is a $\preceq$-ideal, it is stable. 

Let $\Dist(G)$ be the distribution algebra of $G$. For an poset ideal $\Gamma$ in the poset of $p$-regular weights,  let $I_\Gamma$ be the annihilator in $\Dist(G)$
of the category $\sC[\Gamma]$. Then $A_\Gamma:=\Dist(G)/I_\Gamma$ is a quasi-hereditary algebra with
weight poset $\Gamma$ such that $A_\Gamma{\text{\rm --mod}}\cong\sC[\Gamma]$. If $u'$ denotes the sum of the regular
block subalgebras of the restricted enveloping algebra $u$ of $G$, there is a natural homomorphism $u'\to A_\Gamma$
arising because $u$ is a subalgebra of $\Dist(G)$.   If $M$ is a (finite dimensional) $A_\Gamma$-module, it can thus be regarded as
a module for $u$, and hence as a $u'$-module, usually
denoted $M|_{u'}$ or just $M$.

\subsection{The length function.} We will need a ``length function" $\ell:X_{\text{\rm reg}}(T)\to\mathbb Z$ defined, as described below, on the set
of $p$-regular weights.  For this, we follow  \cite[(3.12.3a)]{IKL},
using Lusztig's alcove distance function.   That is, write $\tau
=z\cdot\lambda$, for $\lambda\in C^{+}$, and set $\ell(\tau):=d(C^{+},z\cdot
C^{+})$, which counts, with signs, the number of alcove geometry hyperplanes separating
$C^{+}$ and $z\cdot C^{-}.$  (A ``$+1$" contribution occurs for a hyperplane separating
alcoves $A,B$, in computing $d(A,B)$, when $A$ is on the negative side of the
hyperplane, and a ``$-1$" contribution is used in the opposite case.)  In general,
$\ell(\tau)\not =\ell(z)$ (the Bruhat-Chevalley length for the Coxeter group
$W_{p}$) but the two lengths do agree if $z$ is dominant, i.~e., $z\cdot
C^{+}$ is a dominant alcove. If $z$ is any element of the affine Weyl group $W_p$
or extended affine Weyl group $\widetilde W_p$ associated to $G$, written as the composition $t_{p\theta}w$ of a
translation by $p\theta$ with $\theta\in X(T)$ and $w\in W$ (the Weyl
group), then \cite[Lemma 3.12.5]{IKL}

\[
\begin{aligned}
d(C^{+},z\cdot C^{+})  & =-\ell(w)+2\htt(\theta)\text{, where}\\
2\htt(\theta)  & =%
{\textstyle\sum\limits_{\alpha>0}}
(\theta,\alpha^\vee).
\end{aligned}
\]
When $\theta$ is in the root lattice ${\mathbb Z}R$, the (integer) expression
$2\htt(\theta)$ is an even integer. (This is an easy calculation with the dual
root system.) Consequently, for $\theta\in{\mathbb Z}R$ and $\tau\in
X_{\text{\rm reg}}(T)$,

\begin{equation}
\label{length}
\ell(\tau)\equiv\ell(\tau+p\theta) \,\,\mod 2
\end{equation}

\subsection{The category of $G_1T$-modules.} Following a notation used \cite{AJS},
let $\ssC_k$ be the category of finite dimensional rational $G_1T$-modules. It is a highest weight category
with weight poset $X(T)$. For $\gamma\in X(T)$, let $L_k(\gamma)$ be the irreducible $G_1T$-module of highest weight $\gamma$. In case $\gamma\in X(T)_+$, write $\gamma=\gamma_0+p\gamma_1$ with
$\gamma_0\in X_1(T)$ (the restricted weights) and $\gamma_1\in X(T)_+$. Then $L_k(\gamma)\cong L(\gamma_0)|_{G_1T}\otimes
p\gamma_1$. When $\gamma=\gamma_0$, we usually denote $L_k(\gamma)$ simply by $L(\gamma)$.

 Suppose that $\Omega$ is a union of $W_p$-orbits in $X(T)$. Let $\ssC_k(\Omega)$ be the full subcategory of $\ssC_k$ generated by the irreducible modules $L_k(\gamma)$, $\gamma\in\Omega$. Then $\ssC_k(\Omega)$ is a highest weight category with standard (resp., costandard) modules
denoted $Z_k(\gamma)$ (resp., $Z_k'(\gamma)$) defined by 
\begin{equation} \begin{cases} (1)\,\, Z_k'(\gamma):=\ind_{B_1T}^{G_1T}\gamma; \\ (2)\,\,Z_k(\gamma):=
\ind_{B^+_1T}^{G_1T}(\gamma -2(p-1)\rho).\end{cases}\end{equation}
The category $\ssC_k$ has a natural duality $D$ and $Z_k(\gamma)\cong DZ'_k(\gamma)$.  

 Now assume that $p>h$. Fix a
$p$-regular weight $\lambda\in C^+$.  

A weight $\lambda\in X_{\text{\rm reg}}(T)_+$ will be said to satisfy the Kazhdan-Lusztig property (with respect to the length function $\ell$, defined
in \S2.2),  provided that
\begin{equation}\label{KLproperties} \forall \mu\in X(T), n\in{\mathbb N}, \quad
  \begin{cases} \Ext^n_{G_1T}(L(\lambda),Z_k'(\mu))\not=0\implies n\equiv \ell(\lambda)-\ell(\mu)\,\,\mod \,2; \\ \Ext^n_{G_1T}(Z_k(\mu),L(\lambda))\not=0\implies n\equiv \ell(\lambda)-\ell(\mu)\,\,\mod \,2.\end{cases}
\end{equation}
In fact, it is enough to check this for $\mu\in W_p\cdot\lambda$.

\subsection{The modules $\rDelta(\lambda)$, $\rnabla(\lambda)$; character formulas.} Given $\gamma=
\gamma_0+p\gamma_1\in X(T)_+$, with $\gamma_0\in X_1(T)$  and $\gamma_1\in X(T)_+$, define 
\begin{equation}\label{pmodules}
\begin{cases}\Delta^p(\gamma):=L(\gamma_0)\otimes\Delta(\gamma_1)^{[1]};\\ 
\nabla_p(\gamma):=
L(\gamma_0)\otimes \nabla(\gamma_1)^{[1]}.\end{cases}\end{equation}
 (Given $V\in\sC$, $V^{[1]}\in\sC$ denotes the twist of
$V$ through the Frobenius morphism $F:G\to G$.) The module $\Delta^p(\gamma)$ (resp., $\nabla_p(\gamma)$) is a homomorphic image (resp., submodule) of $\Delta(\gamma)$ (resp., $\nabla(\gamma)$)
and so is indecomposable with head (resp., socle) $L(\gamma)$. 

There is another family of rational $G$-modules, denoted $\rDelta(\gamma)$ and $\rnabla(\gamma)$, $\gamma\in\Gamma$, which are
closely related to the modules above. These modules are obtained, by a standard ``reduction mod $p$,"
from the irreducible type 1 modules $L_\zeta(\gamma)$, $\gamma\in X(T)_+$, for the quantum enveloping
algebra $U_\zeta$ associated to the root system $R$ at a primitive $p$th root of unity $\zeta$. This is
described in detail in \cite[\S2]{PS11}, which contains other references to the literature. We do not
repeat this, except to let $\widetilde U_\zeta$ be the Lusztig $\sO$-form associated to
$R$ in which each $K_i^p=1$.  Then $\wU_\zeta\otimes_{\sO}K\cong U_\zeta$. If $\overline U_\zeta:=\wU_\zeta/\pi\wU_\zeta$ and if $I$ is the ideal of generated by the $K_i-1$, $i=1,\cdots, n$, where $n$ is the
rank of $R$, then $\overline U_\zeta/I \cong \Dist(G)$. In this way, $U_\zeta$-modules which are 
integrable and of type 1, can be ``reduced mod $p$" to obtain rational $G$-modules. Thus, given
$\gamma\in X(T)_+$, $\rDelta(\gamma)$ (resp., $\rnabla(\gamma)$) is the rational $G$-module obtained
by reduction mod $p$ of the irreducible module $L_\zeta(\gamma)$ for $U_\zeta$ using a minimal (resp.,
maximal) admissible lattice.
  (Sometimes, it will be convenient
to write $L_K(\lambda)$ for $L_\zeta(\lambda)$, $\Delta_K(\lambda)$ for $\Delta_\zeta(\lambda)$, etc.)

In this paper, we will say that  the Lusztig character formula (LCF) holds for a finite poset ideal $\Gamma$ of $p$-regular dominant
weights provided that 
\begin{equation}\label{LCF}\rDelta(\gamma)=\Delta^p(\gamma),\,\,\forall\gamma\in\Gamma.\end{equation}
 (Equivalently, this means that 
 $\rnabla(\gamma)=\nabla_p(\gamma)$, for all $\gamma\in \Gamma$.) By
\cite[Cor. 2.5]{PS11}, if $\Gamma$ is the ideal generated by the $p$-regular restricted weights, then
this condition implies that the characters of the irreducible modules $L(\gamma)$, for $\gamma\in X_1(T)$
are all given by Lusztig's character formula \cite{L}. If the poset $\Gamma$ is stable in the sense of
\S2.1, this just means that (\ref{LCF}) holds for $\gamma\in\Gamma$ a restricted dominant weight or,
equivalently , $L_\zeta(\lambda)$ and $L(\lambda)$ have the same dimension for $\lambda\in\Gamma$
restricted. Observe the existence of $p$-regular weighs means that $p\geq h$. Often, we will assume
that $p$ is larger, e.~g., $p\geq 2h-2$.

Finally, we say  the ``LCF holds" (not mentioning any poset) to mean that (\ref{LCF}) holds for
all $p$-regular dominant weights, or, equivalently, for all restricted dominant weights.\footnote{This formulation
makes sense for $p\geq h$, though generally we assume $p>h$, where \cite[p. 273]{T} guarantees that
each $L_\zeta(\lambda)$ satisfies Lusztig's character formula for all dominant weights $\lambda$.}

 \subsection{Grading the restricted enveloping algebras.}
The category of rational $G_{1}$-modules is equivalent to the category of modules for
the restricted enveloping algebra $u$ associated to $G.$ (For this reason, we freely identify
$G_1$-modules with $u$-modules in the discussion.) The maximal torus $T$
acts rationally (via the adjoint action) on $u$ as automorphisms $x\mapsto {^{t}x}$ ($t\in T, x\in u$).
This action induces a familiar (weight space) decomposition $u=\bigoplus u_{\nu }$
in terms of all $\nu \in {\mathbb Z}R \cup \{0\}\subseteq X$ where $X:=X(T)$. 
More generally, every rational $T$-action  on a finite dimensional vector
space $V$ over $k$ has a (direct sum) decomposition $V=\bigoplus _{\tau \in X}V_{\tau }$. There
is, of course, a converse, but that is not what we wish to emphasize. 
Following the work of Andersen-Jantzen-Soergel in \cite[Appendix E]{AJS}, any
decomposition $V=\bigoplus _{\tau \in X}V_{\tau }$ is called an $X$--grading on a vector space $V$. This
makes sense for any abelian group $X$, though we focus on the special case $
X=X(T).$ From this point of view, a $G_{1}T-$module is a $u$-module
equipped with an $X-$grading that satisfies certain compatibility
conditions. These are the multiplication  conditions $u_{\nu }V_{\tau
}\subseteq V_{\nu +\tau }$ $(\nu ,\,\,\tau \in X)$ and another
requirement.\footnote{Namely, it is required that $h_\alpha\cdot v=(\tau,\alpha^\vee)v$ for $\tau\in X$, $v\in 
V_\tau$.}  which takes into account the fact that part of $T$ is
inside $G_1$. See \cite[2.4]{AJS}, which also gives a discussion
for the analogous quantum situation (at a root of unity) using the $X$-grading terminology.\footnote{In the
quantum case, it is required that $K_\alpha\cdot v=\zeta^{(\tau,\alpha^\vee)} v$, $v\in (u_\zeta)_\tau$. The small
quantum group $u_\zeta$ is also $X$-graded.} This gives, among other things, a useful uniformity
of terminology.   We are particularly concerned with (positive) 
$\mathbb{Z}$-gradings on either $u_\zeta$ and $u$ which might be compatible with their respective 
$X$--gradings. (To say that a space or algebra $V$ over $k$ with an $X$-grading
has a compatible $Z$-grading $V=\bigoplus _{n\in \mathbb{Z}}V_{n}$ just means
that each $V_n$ is the sum of its intersections with the various spaces 
$V_\tau$. This is equivalent to the compatibility notion  in  
\cite[F.8]{AJS}.)  Also,  \cite[18.21]{AJS} observes that every block algebra
component $\sB$ of either algebra carries a natural $X$-grading. The same subsection shows in 
\cite[Prop. 18.21 \& Rem. 18.21(2)]{AJS}, under the validity of the LCF, that these block
algebras, when $p$-regular, carry a compatible Koszul grading.\footnote{In fact, in the quantum cases, $p$ is not required
to be a prime, but does need to satisfying some other conditions---see \cite[p. 231]{AJS}---all of which hold if
$p$ is a prime $>h$.} (The regularity
requirement just means that the irreducible modules in the block are
parameterized by $p$-regular weights.) We will need to quote this result and its proof 
 in the proofs of Theorem \ref{thm2.1} and Theorem \ref{thm2.2} below.\footnote{The brief proof given in 
\cite[18.21]{AJS} ignores the nontrivial relationship between the $X$-weights on 
$\sB$ and those arising when $\sB$ is considered as an
endomorphism algebra.  We supply the needed discussion in the Appendix below and
explain how it completes the proof.} The following result for the quantum
case is essentially a special case (for $p$ a prime) of this result of \cite[\S18.21]{AJS}.

\begin{thm}\label{thm2.1}
Suppose $p>h$ is a prime and $\zeta $ is a primitive $p$th root of unity. Then
any regular block $\sB$ of $u_{\zeta }$ has a Koszul
grading compatible with its $X$--grading.\end{thm}

\begin{proof}
Given that the LCF always holds at the $p$th roots of unity
quantum case  as long as $p>h$ (see \cite{T}), the theorem is an immediate
consequence of \cite[Prop.18.21 \& Rem. 18.21(2)]{AJS}.
\end{proof}

As an easy consequence, the same theorem holds if $
\sB$ is replaced by the sum $u'_\zeta$ of
all regular block components of the algebra $u_\zeta.$ Let $(K,\sO,k)$ be the $p$-modular
system of \S2.4. The usual $\sO$-form $\widetilde{u}_\zeta'$ of $u'_\zeta$  has a positive grading that
base changes (i.~e., by applying $-\otimes_\sO K$) to a Koszul grading on $u'_\zeta$ \cite[Thm. 8.1]{PS12}.
The latter result states that the Koszul grading on $u'_\zeta$ is that obtained in \cite[\S\S17--18 \& p. 231]{AJS}.
The reference to p. 231 implicitly refers to Conjecture 2 on that page which mentions the X-grading compatibility, established
for all regular blocks of $u'_\zeta$.  The result in \cite{AJS} which exhibits such an $X$-grading
is \cite[Prop. 8.21 \& Rem. 18.21(2)]{AJS}. The reader may confirm by comparing the proof of this latter
result with that of \cite[Thm. 8.1]{PS12} that they use the same Koszul grading on $u'_\zeta$. 
The discussion of $X$-gradings versus $Y$-gradings ($Y:=p{\mathbb Z}R$) above \cite[Prop. 18.21]{AJS}  can be replaced by the Appendix to this paper.

Since
the grading on $\widetilde u_\zeta'$ given in \cite[\S8]{PS12} base changes to
that on $u'_\zeta$, it can itself be obtained by taking intersections
with the grading on $u'_\zeta$. This is also true for its natural $X$-grading. This
proves the analog of Theorem \ref{thm2.1} stated below as Theorem \ref{thm2.2} for $\widetilde u_\zeta'$. The algebra $u'$ is the sum of all regular blocks of $u$, and it
it is the reduction mod $p$ of $\widetilde u_\zeta'.$

\begin{thm}\label{thm2.2}
Suppose $p>h$ is a prime and $\zeta $ is a primitive $p$th root of unity.
Then the algebra $\widetilde u_\zeta'$ has a positive integer grading, compatible with its $X$-grading, which bases changes to a Koszul grading on $u'_\zeta$ also compatible with its $X$-grading. Applying $-\otimes_\sO k$, this grading on $\widetilde u'_\zeta$ also base changes to a
positive grading on $u'$ compatible with its $X$-grading (as induced by the adjoint action of $T$).
\end{thm}

In the statement of the theorem, there is no claim about the Koszulity of the positive grading on $u'$, though
this will be true (as follows from \cite{AJS}) when the LCF holds for $p$-restricted weights.

\section{A review of some earlier results} In this section, we briefly review some results obtained in \cite{PS11} and \cite{PS13}. For convenience, let $\wfa=\widetilde u_\zeta'$, and for any integer $n\geq 0$, define
$\wrad^n\wfa:=\wfa\cap\rad^n\wfa_K$. Then we set, for any $\wU_\zeta$-module $M$,
\begin{equation}\label{wgr}\wgr M:=\bigoplus_{n\geq 0}(\wrad^n\wfa)M/(\wrad^{n+1}\wfa)M.\end{equation}
In particular, we can take $M:=\wfa$, to obtain an algebra $\wfa$ over $\sO$, and, for any $M$, $\wgr M$
is naturally a $\wgr \wfa$-module. More generally, 
let $\Gamma$ be a finite poset ideal of $p$-regular dominant weights. If $p\geq 2h-2$,  by  \cite[Thm. 6.1]{PS10}, the
graded algebra $\wgr A_\Gamma$ is quasi-hereditary with weight poset $\Gamma$.  Additionally, $\wgr A_\Gamma$ has standard (or Weyl) modules of the form $\wgr \Delta(\gamma)$, $\gamma\in\Gamma$.\footnote{The notation in \cite{PS10} is slightly different than that used here in that we write $\wgr A_\Gamma$ more
simply as $\gr A_\Gamma$ (which has the danger of being confused with the radical series of $A_\Gamma$,
from which it may differ. Also, \cite{PS10} proves a much stronger result which states that the quasi-hereditary algebra $\wgr A_\Gamma$ arises through base change from a quasi-hereditary algebra $\wgr \wA_\Gamma$
over $\sO$.  We will not need that here.} These were studied in \cite{PS13} under stronger hypotheses
which we will assume here:
 
\begin{quote}
{\it A
standing hypotheses in the remainder of this section is that the  LCF (as recast in (\ref{LCF}))  holds and that $p\geq 2h-2$ is odd. }Throughout,
$\Gamma$ will be a fixed ideal of $p$-regular dominant weights.\end{quote}

\begin{thm}\cite[\S5]{PS11}\label{Weyl} Given any dominant weight $\lambda$, the Weyl module $\Delta(\lambda)$ has a filtration 
by $G$-submodules with corresponding sections of the form $\Delta^p(\gamma)$, $\gamma\in X(T)_+$. In case $\lambda$
is $p$-regular, this filtration can be taken to be compatible with the $G_1$-radical series of $\Delta(\lambda)$, in the sense that
each section of the radical series has a $\Delta^p$-filtration.\end{thm}

\begin{thm} \label{theorem5.3} \cite[Thm. 5.3(a)]{PS13} Let $\lambda,\mu\in X_{\text{\rm reg}}(T)_+$.  
For any integer $n\geq 0$,  the rational $G$-module $\Ext^n_{G_1}(\rDelta(\lambda),\rnabla(\mu))^{[-1]}$  has a $\nabla$-filtration.\end{thm}

Before stating the next result, we recall some standard terminology. In case $B$ is a graded algebra and $M,N$ are graded $B$-modules, $\grExt_B^\bullet(M,N)$ denotes
the Ext-groups computed in the category of graded $B$-modules. If $r\in\mathbb Z$, then $N\langle r\rangle$
is the graded $B$-module obtained from $N$ by shifting the grades $r$-steps to the right, i.~e.,
$N\langle r\rangle_i:=N_{i-r}$, for all $i\in\mathbb Z$. Therefore,
\begin{equation}\label{ext}\Ext^n_B(M,N)=\bigoplus_{r\in\mathbb Z}\grExt^n_B(M,N\langle r\rangle).
\end{equation}. 
We will use these conventions in the next result. 

\begin{thm}\label{firstmainconsequencethm5.6}\cite[Thm. 5.6]{PS13}  Let $\lambda,\mu\in\Gamma$. 
$$\forall n\in{\mathbb N}, r\in{\mathbb Z},\quad\grExt^n_{\wgr A}(\rDelta(\lambda),\rnabla(\mu)\langle r\rangle)\not=0\implies r=n.$$ \end{thm} 

Recall from \cite[Defn. \ref{firstmainconsequencethm5.6}]{PS13} that a positively graded algebra $B$ is called a
{\it Q-Koszul algebra} provided that:

\begin{itemize} \item[(1)] its grade 0 component $B_0$ is quasi-hereditary with poset
$\Gamma$ (and with standard and costandard modules denoted $\Delta^0(\gamma)$ and
$\nabla_0(\gamma)$ (respectively); and 

\item[(2)] if $\Delta^0(\gamma)$ and $\nabla_0(\gamma)$ are given
pure grade 0, as graded $B$-modules, then 
\begin{equation}\label{QKoszul}\forall n\in{\mathbb N}, r\in{\mathbb Z},\lambda,\mu\in\Gamma, \quad\grExt^n_B(\Delta^0(\lambda),\nabla_0(\mu)\langle r\rangle)\not=0\implies n=r.\end{equation}
\end{itemize}

Suppose that $B$ is Q-Koszul {\it and} a graded quasi-hereditary algebra with weight poset $\Gamma$, having graded
standard (resp., costandard) modules $\Delta^B(\gamma)$ (resp., $\nabla_B(\gamma)$), $\gamma\in\Gamma$, with head (resp., socle) of grade 0. If
\begin{equation}\label{conditions1}\begin{cases}\grExt^n_B(\Delta^B(\lambda),\nabla_0(\mu)\langle r \rangle )\not=0\implies n=r;\\
\grExt^n_B(\Delta^0(\mu),\nabla_B(\lambda)\langle r \rangle )\not=0\implies n=r,\end{cases}\end{equation}
then we say that $B$ is a {\it standard} Q-Koszul algebra.

The notions of Q-Koszul and standard Q-Koszul algebras are generalization of the notions of
Koszul and standard Koszul algebra.  In the terminology above, a Koszul algebra $B$ is a Q-Koszul
algebra in which the modules $\Delta^0(\gamma)$ and $\nabla_0(\gamma)$ are all irreducible. 
Equivalently, $B_0$ is semisimple.
An algebra $B$ is a {\it standard Koszul algebra} if it is a Koszul algebra and a graded quasi-hereditary
algebra,\footnote{A graded quasi-hereditary algebra is just a quasi-hereditary algebra with a positive
grading. All irreducible, standard and costandard modules (and more) will have graded versions as above; see
\cite{CPS1a}. Here the positive grading is taken from the Koszul algebra.}
 and if the conditions (\ref{conditions1}) hold.
 (See comments above \cite[Defn. 3.6]{PS13} for some history of the ``standard" Koszul terminology. Our usage comes from Mazorchuk \cite{Maz}. For more history of Koszul gradings,  see \cite{PS9}.)

\begin{thm}\label{filtration6.2}\cite[Thm. 6.2]{PS13}   For $\mu,\nu\in\Gamma$, the rational $G$-module $\Ext^m_{G_1}(\Delta(\nu),\rnabla(\mu))^{[-1]}
$ has a $\nabla$-filtration and the restriction natural map
\begin{equation}\label{Jantzentheorem}
\Ext^m_{G_1}(\rDelta(\nu),\rnabla(\mu))\to\Ext^m_{G_1}(\Delta(\nu),\rnabla(\mu))\end{equation}
 is surjective.
 
 Dually, the rational $G$-module $\Ext^m_{G_1}(\rDelta(\mu),\nabla(\lambda))^{[-1]}$ has a $\nabla$-filtration
 and the natural map 
 \begin{equation}\label{Jantzentheorem2}
 \Ext^m_{G_1}(\rDelta(\nu),\rnabla(\mu))\to\Ext^m_{G_1}(\rDelta(\nu),\nabla(\nu))\,\,
 {\text{\rm is surjective}}.\end{equation}
\end{thm}

The following
theorem implies that $\wgr A_\Gamma$ is a standard Q-Koszul algebra.

\begin{thm}\label{secondmainconsequencethm3.7} \cite[Thm. 3.7]{PS13} Let $\lambda,\mu\in\Gamma$. For any nonnegative integer $n$
and any integer $r$,
$$\grExt^n_{\wgr A}(\wgr\Delta(\lambda),\rnabla(\mu)\langle r \rangle )\not=0\implies r=n$$
and
$$\grExt^n_{\wgr A}(\rDelta(\lambda),\nabla_{\wgr A}(\mu)\langle r \rangle )\not=0\implies r=n,$$
where $\nabla_{\wgr A}(\mu)$ is the costandard object in the highest weight category
corresponding to $\mu$.\end{thm}

 If $\Gamma$ is a poset of $p$-regular weights contained in the Jantzen region $\Jan$ and if the LCF holds, then $\wgr A_\Gamma
 \cong\gr A_\Gamma$, the graded algebra obtained from the radical filtration of $A_\Gamma$. Similarly, $\wgr\Delta(\gamma)\cong\gr\Delta(\gamma)$, the $\gr A_\Gamma$-module obtained from the radical series of $\Delta(\gamma)$.  

\begin{cor}\label{corollary3.8} \cite[Cor. 3.8]{PS13} Now assume that $\Gamma$ is contained in the Jantzen region $\Jan$. Then  $\gr \Delta(\lambda)$ is a linear module over $\gr A$. Also, the graded quasi-hereditary algebra $\gr A$-mod has a graded Kazhdan-Lusztig theory. In particular, $\gr A$ is Koszul. \end{cor}

The ``graded Kazhdan-Lusztig theory" property implies that $\gr A$ is standard Koszul. See \cite{CPS1}.

\section{Some new properties of Koszul and Q-Koszul algebras}

Some of the results of this section are quite general, and the characteristic of the underlying algebraically
closed field $k$ may be arbitrary, unless a prime $p$ is mentioned (in which case, $p$ is the characteristic of
$k$). Our first result formalizes in the Q-Koszul algebra case a property observed for Koszul algebras
in \cite[Proof of Cor. 3.2]{ADL}. Suppose that $M$ is a non-negatively graded module for a Q-Koszul
algebra $B$ (as defined after Theorem \ref{firstmainconsequencethm5.6} above). Then we will say that $M$ is {\it Q-linear} provided that
\begin{equation}\label{Qlinearmodule}\forall n\in{\mathbb N},r\in{\mathbb Z},\gamma\in\Gamma,\quad
\grExt^n_B(M,\nabla_0(\gamma)\langle r\rangle)\not=0\implies n=r.\end{equation}
A non-positively graded $B$-module $M$ is called {\it Q-colinear} provided that 
\begin{equation}\label{Qcolinearmodule}\forall n\in{\mathbb N}, r\in{\mathbb Z},\gamma\in\Gamma,\quad
\grExt^n_B(\Delta^0(\gamma)\langle -r\rangle,M)\not=0\implies n=r.\end{equation}

\begin{prop}
\label{prop4.1} Let $B$ be a Q-Koszul algebra with weight poset $\Gamma$.
Let $M$ be a non-negatively graded Q-linear $B$-module. Assume for each $
i\geq0$, that $M_{i}$ (regarded as a $B_{0}\cong B/B_{\geq 1}$-module) has a $\Delta^{0}$-filtration. Then each 
$$M_{\geq i}\langle-i\rangle:=\left(\bigoplus_{j\geq i}M_{j}\right)\langle -i\rangle$$
is a Q-linear $B$-module.
\end{prop}

\begin{proof}
We proceed by induction on $i$. Since $M_{\geq0}\langle0\rangle=M$, the
statement is true for $i=0$ because $M$ is assumed to be Q-linear. Now fix 
$i\geq0$ and assume that $M_{\geq i}\langle-i\rangle$ is Q-linear. We will
show that $M_{\geq i+1}\langle-i-1\rangle$ is Q-linear. The short exact
sequence $0\to M_{\geq i+1}\to M_{\geq i}\to M_{i}\to0$ of $B$-modules gives, for any
$\mu\in\Gamma$, 
a long exact sequence 
\begin{equation*}
\cdots\to{\text{\textrm{Ext}}}_{B}^{n}(M_{i},\nabla_{0}(\mu))\overset{\alpha 
}{\longrightarrow}{\text{\textrm{Ext}}}^{n}_{B}(M_{\geq i},\nabla_{0}(\mu)) 
\overset{\delta}{\longrightarrow}{\text{\textrm{Ext}}}^{n}_{B}(M_{\geq
i+1},\nabla_0(\mu))\longrightarrow\cdots 
\end{equation*}
We claim the mapping $\alpha$ is surjective, or equivalently $\delta=0$.
Assume not, so that 
\begin{equation*}
\text{\textrm{ext}}^{n}_{B}(M_{\geq i+1},\nabla_{0}(\mu)\langle n\rangle
)\not =0, 
\end{equation*}
for some $\mu\in\Gamma$. Because $M_{\geq i+1}$ has a $\Delta^{0}$-filtration,
it follows, for some $s\geq1$, that 
$\grExt_{B}(\Delta_{0}(\gamma)\langle s\rangle,\nabla_{0}(\mu)\langle n\rangle
)\not =0$. Thus, $\grExt^n_{B}(\Delta^{0}(\gamma),\nabla
_{0}(\mu)\langle n-s\rangle)\not = 0$, contradicting the assumption that $B$
is Q-Koszul.

We conclude that ${\text{\textrm{Ext}}}^{n}_{B}(M_{\geq i+1},\nabla_{0}
(\mu))\subseteq{\text{\textrm{Ext}}}^{n+1}_{B}(M_{\geq i},\nabla_{0}(\mu))$, for all $\mu\in\Gamma$.
Hence, for any integer $m\geq0$, if 
$$\text{\textrm{ext}}^{n}_{B}(M_{\geq
i+1}\langle-i-1\rangle,\nabla_{0}(\mu)\langle m\rangle)\not =0,$$
 then $\grExt^{n+1}_{B}(M_{\geq i}\langle-i\rangle, \nabla_{0} (\mu)\langle
m+1\rangle)\not =0$. Thus, $n=m$, as required.
\end{proof}

A similar result holds for a non-positively graded Q-colinear module. The same is true for Corollaries
\ref{cor4.2} and \ref{cor4.3} below, and their generalizations at the end of \S6. We leave further details to the reader.

Now we return to the situation of the representation theory of our group $G$. Let $\Gamma$ be a finite poset ideal of $p$-regular weights, and consider
the graded quasi-hereditary algebra $B:=\wgr A_{\Gamma }$. It has standard modules $\Delta^{B}(\gamma):=\wgr\Delta(\gamma)$, $\gamma\in\Gamma$. Its costandard modules
are denoted $\nabla_{B}(\gamma)$. In addition, the grade 0 component of $
B_{0}$ of $B$ is quasi-hereditary with weight poset $\Gamma$ and with
standard (resp., costandard) modules $\Delta^{0}(\gamma):=\rDelta(\gamma)$ (resp., 
$\nabla_{0}(\gamma):=\rnabla(\gamma)$), $\gamma\in\Gamma$. If we assume the LCF holds, then the
algebra $B$ is standard Q-Koszul. Put $\Delta^B_i(\gamma):=\Delta^B(\gamma)_{\geq i}$, for
each integer $i\geq 0$. 

\begin{cor}\label{cor4.2}
Assume that $p\geq 2h-2$ is odd and that the LCF holds. If $\gamma\in\Gamma$
and $i\geq0$, $\Delta_{i}^{B}(\gamma)\langle -i\rangle$ is Q-linear. (Here $B:=\wgr A_{\Gamma}$.)
\end{cor}

\begin{proof}
By Theorem \ref{firstmainconsequencethm5.6}, $\Delta^{B}(\gamma)$ is a Q-Koszul module. By Theorem \ref
{Weyl}, each 
\begin{equation*}
\Delta^{B}_{i}(\gamma)/\Delta^{B}_{i+1}(\gamma) 
\end{equation*}
has a $\Delta^{0}$-filtration. Thus, the hypotheses of Proposition \ref
{prop4.1} hold, and the proof is complete.
\end{proof}
 
\begin{cor}\label{cor4.3}
Assume that $p\geq 2h-2$ is odd and that the LCF holds. Let $\Gamma$ be a
poset ideal of $p$-regular dominant weights which is contained in the
Jantzen region $\Jan$.\footnote{$\Jan$ contains all restricted weights if and only if $p\geq 2h-3$.} Then $B:=\gr A_{\Gamma}$ is a
Koszul algebra and, given any $\gamma\in\Gamma$ and $i\geq 0$, the module
$\Delta _{i}^{B}(\gamma)$ is linear for $B$. In particular, both $\gr\Delta(\gamma)$ and its maximal submodule (shifted by $\langle-1\rangle$)
are linear modules for $\gr A_\Gamma$.
\end{cor}

Maximal submodules of standard modules are especially interesting for the
study of the associated irreducible modules. The above corollary shows that
their ext groups with coefficients in irreducible modules are
especially well-behaved. We can prove a similar property for ext groups of these
modules (and any term of their radical series) with coefficients in
costandard modules. It is useful to discuss this before stating the next
theorem.

Let us say that a graded module $M$ for a standard Koszul algebra $B$ (with weight poset
$\Gamma$) is
{\it strongly linear} if the following property holds:
\begin{equation}\label{strong} \forall \gamma\in\Gamma, n\in{\mathbb N}, r\in{\mathbb Z}, \quad
\grExt_{B}^n(M,\nabla
_{B}(\gamma)\langle r\rangle )\neq 0\implies n=r.\end{equation}
There is an evident dual
notion of a strongly colinear module. By definition, the  (purely graded) irreducible
modules for $B$ are always strongly linear and strongly colinear. 

If 
$\Omega $ is a coideal in $\Gamma$, the stronge linearity  property of any
module is preserved upon passage to (graded versions of ) the natural
highest weight category associated to $\Omega$, and a similar statement
holds for  the strong colinearity property. In more detail,  the passage is
obtained by an exact additive functor $j^*:B{\text{\rm -mod}}\to eBe{\text{\rm -mod}}$, $M\mapsto j^*M=eM$, obtained by multiplication by a grade 0 idempotent $e\in B$. The functor $j^*$ maps standard (resp.,
costandard) modules $\Delta^B(\gamma)$ (resp., $\nabla_B(\gamma)$) for $\gamma\in\Omega$ to the
corresponding standard and costandard modules in $eBe$-mod. In addtion, the functor $j^*$ admits a left exact right adjoint $j_*:=\Hom_B(eB,-)$ which carries any costandard
module $\nabla_{eBe}(\gamma)$, $\gamma\in\Omega$, in $eBe$-mod to the corresponding costandard module $\nabla_B(\gamma)$ in $B$-mod. Thus, for any $B$-module $E$ and $\gamma \in \Omega $, $j^*$ induces
an isomorphism 
\begin{equation*}
\Ext_{B}^{n}(E,\nabla _{B}(\gamma ))\overset\sim\rightarrow \Ext_{eBe}^{n}(eE,e\nabla _{B}(\gamma ))\cong 
\Ext_{eBe}^{n}(eE,\nabla _{eBe}(\gamma ))
\end{equation*}
  Dually, the strong colinearity property is similarly
preserved by $j^*$ (which admits a right exact adjoint $j_!$ taking costandard modules to costandard modules). As one consequence (using both strong linearity properties of irreducible
modules), we can deduce that standard and costandard modules $eBe$-mod are linear and colinear,
respectively, which implies that the algebra $eBe$ is standard Koszul, and, in particular, Koszul.\footnote{This Koszulity implication goes back to 
Irving \cite{Irving}, as discussed in \cite[p. 345]{PSZ2}. It may also deduced from graded Grothendieck
group arguments, as in \cite[\S3, appendix]{CPS1}. An ungraded analogue is given in 
\cite[Thm. 1]{ADL}.}  As another consequence of
the displayed isomorphism, we can deduce that,  
{\it for any strongly linear $B$-module $M$, the modules $M_{\geq
i}\langle -i\rangle$ are also strongly linear.}
This is seen by choosing, for a given $\gamma \in \Gamma$, a coideal $
\Omega $ with $\gamma $ minimal in $\Omega$. The minimality implies 
$\nabla_{eBe}(\gamma )$ is irreducible. Now the argument of Proposition \ref{prop4.1} can be
applied to $eBe$ for this fixed $\gamma $, using the modules $e(M_{\geq
i})=(eM)_{\geq i}$, to inductively deduce the strong linearity. The dual
property, for strongly colinear $B$-modules, may be deduced by a dual
argument.

In particular, Corollary \ref{cor4.3} holds if ``linear" is replaced by ``strongly linear." Explicitly,

\begin{prop}\label{stlin} Assume the hypotheses of Corollary \ref{cor4.3}. For $\gamma\in\Gamma$ and $i\geq 0$, 
 each $\Delta_i^B(\gamma)$ is strongly linear. A dual statement holds for costandard modules, using
 strong colinearity. \end{prop}

This strong linearity property, for maximal submodules of standard modules (and their
radical series, each appropriately shifted in grade) is new. Here is another new result for standard Koszul
algebras, applicable to maximal submodules of standard modules and their
radicals series on the ``strong linearity" side, and to dual notions for
costandard modules on the ``strong colinearity" side.

\begin{thm}
Suppose $B$ is a standard Koszul algebra. Let $M$
(resp., $N$) be a strongly linear (resp., strongly colinear) module for $B$. 
Then, for all integers $n$ and $r$, $\grExt_{B}^{n}(M,N\langle r\rangle )\not=0\implies n=r$.
\end{thm}

\begin{proof}
We give the proof only in a special case, which is likely to be more familiar. First, assume that $B$-mod has a Kazhdan-Lusztig theory
in the sense of \cite{CPS1}. If $\Gamma$ is the poset for $B$, this supposes there is a length function $\ell:\Gamma\to
\mathbb{Z}$, used mod 2 to assign parities to modules indexed by elements $\gamma\in\Gamma $.
More explicitly, it is required that 
$$\forall n\in{\mathbb N}, \gamma,\mu\in\Gamma,\quad \begin{cases}\Ext^n_B(\Delta^B(\gamma),L_B(\mu))\not=0\implies n\equiv \ell(\lambda)-\ell(\mu)\,\,\mod 2;\\
\Ext^n_B(L_B(\mu),\nabla_B(\mu))\not=0\implies n\equiv\ell(\lambda)-\ell(\mu)\,\,\mod 2.\end{cases}$$
    (In the
presence of the Koszulity property for $B,$ the existence of such a
Kazhdan-Lusztig theory implies $B$ is a standard Koszul algebra, and many of the known
examples arise this way. See \cite{CPS1}, especially the appendix to \S3 and the argument for Theorem 
2.4.)  Second, in addition to the Kazhdan-Lusztig theory, we
will assume an additional property of $M$ and $N$, namely, that all the irreducible
constituents of the head of $M$ (which may be identified with $M_0$) all
share a common parity (with regard to $\ell$), and that a similar parity sharing occurs for
irreducible constituents of the socle $N_0$ of $N.$ 

These additional conditions can all be 
avoided by using the
 somewhat more sophisticated notion of a ${\mathbb Z}/2$-based Kazhdan-Lusztig theory in \cite{PSZ2}. 

Returning to our chosen context, for $\gamma \in \Gamma 
$, observe that each map 
$$\Ext_{B}^{n}(M_{0},\nabla _{B}(\gamma
))\rightarrow\Ext_{B}^{n}(M,\nabla _{B}(\gamma ))$$ is
surjective. (Equivalently, the map $\Ext_{B}^{n}(M,\nabla
_{B}(\gamma ))\rightarrow\Ext_{B}^{n}(M_{\geq 1},\nabla
_{B}(\gamma ))$ is zero. But this can be deduced by passing to a suitable
algebra $eBe$ with $e\nabla _{B}(\gamma )$ irreducible, and arguing with
natural isomorphism induced by adjoint functors. See the argument in the
paragraph preceding the theorem, and the proof of Proposition \ref{prop4.1}.) This
gives the ungraded groups $\Ext_B^n(M,\nabla
_{B}(\gamma ))$ an even-odd vanishing property, the same as that possesed
by $M_{0}$ or any of its irreducible constituents. A similar even-odd
vanishing property is obtained dually for $N.$ In particular, this yields
the important conclusion (from the derived category arguments of \cite{CPS1}) that $M$ and $N$, respectively, belong to certain filtered derived
subcategories, each associated with a particular parity of length function. 
$M_{0}$ and $M$ belong to the same subcategory $\mathcal{E}^{L}$ or  $
\mathcal{E}^{L}[1]$ and $N_{0}$ belongs to the subcategory, $\mathcal{E}^{R}$
or  $\mathcal{E}^{R}[1],$as $N$. However, $M_{1}$ and $M_{\geq 1}$
belong to the subcategory $\mathcal{E}^{L}[1]$ or $\mathcal{E}^{L}$
associated with the opposite parity to that of $M_{0}$ and $M$.  (We know
from above $M_{\geq 1}\langle -1\rangle$ is strongly linear. Also, $\Ext^1$ nonvanishing
between irreducible modules forces them to have opposite parity. Dual
considerations apply for $N/N_0$ to give it a parity opposite to that of 
$N_0$.) We do not discuss in detail the meaning of these parity differences
other than to note they imply $\Ext_B^n(M,N)$ and $\Ext_{B}^{n}(M_{\geq
1},N)$ cannot be simultaneously nonzero. Also, $\Ext_{B}^{n}(M_{0},N)$ and $\Ext_{B}^{n}(M_{0},N/N_{0})$ cannot be simultaneously nonzero.  

 Next, we prove the theorem for the case $M=M_0$. The theorem
is certainly true in this case, if $N=N_0$. Suppose $\grExt_{B}^{n}(M_{0},N\langle r\rangle)\neq 0$
Then  $\Ext_{B}^{n}(M_{0},N)\neq 0$, so $\Ext_{B}^{n}(M_{0},N/N_{0})=0.$
Consequently, then natural map $\Ext_{B}^{n}(M_{0},N_{0})\rightarrow \Ext_{B}^{n}(M_{0},N)$ is surjective, inducing a surjection 
$\grExt_{B}^{n}(M_{0},N_{0}\langle r\rangle)\rightarrow \grExt_{B}^{n}(M_{0},N\langle r\rangle).$  Hence, it follows that
$\grExt_{B}^{n}(M_{0},N_{0}\langle r\rangle)\neq 0,$ and so $n=r$, in this case.  

 Similarly, the theorem for general $M$ follows from the $M=M_{0}$
case, using the fact that Ext$_{B}^{n}(M,N)$ and Ext$_{B}^{n}(M_{\geq 1 },N)$
cannot be simultaneously nonzero. 

This completes the proof for the case we have chosen. A
general proof along roughly similar line, though working with parity
considerations on ext groups, and appropriate categories  $\sE^{L'}$ and $\sE^{R'}$ may
be obtained using \cite{PSZ2}, but we omit further details.
\end{proof}

Several remarks are in order. First, it is interesting to note the above
theorem implies that "strongly linear" modules are also ``linear," with a
similar property for ``strongly colinear" modules. (That is, these modules are also colinear.)
Second, all of the above results for standard Koszul algebras appear to
generalize to the standard Q-Koszul case, though we have not checked all
details. Third, it is certainly not necessary to assume positive
characteristic in the results above that are stated using a condition on $p$,
and these results hold, mutatis mutandis, for the BGG categories 
$\mathcal O$. In fact, in that case, the algebra $B=\gr A$ is
quasi-hereditary, because it is isomorphic to $A$.

\section{A graded Ext result for $G_1T$} In this section, $\Gamma$ is a stable poset ideal of $p$-regular dominant weights. Suppose the
Kazhdan-Lusztig property (\ref{KLproperties}) holds for all $\gamma\in \Gamma$. Assume that $p>h$. Then, we can adopt an  
argument given in \cite{IKL} to show that if $\lambda,\mu\in \Gamma+pX$, then
\begin{equation}\label{firstcongruence}\Ext^n_{G_1T}(L(\lambda),L(\mu))\not=0\implies \ell(\lambda)-\ell(\mu)\equiv n\,\,\mod \, 2\end{equation}
In more detail, the argument for \cite[Thm. 5.6]{IKL} is an inductive argument on lengths of restricted weights, starting with weights in the lowest dominant alcove $C^+$. Each $W_p$-orbit (under the "dot" action) of $p$-regular weights must contain such a weight, as will any nonempty intersection of such an orbit of $\Gamma$. Then the inductive argument works entirely with restricted weights, increasing their lengths by 1 at each step of the argument. Every restricted weight is accessible in such a process. The stability assumption on $\Gamma$ guarantees that, whenever any one of its restricted weights is accessed by such a sequence, each element $\nu$ of the accessing sequence also belongs to $\Gamma$. This just gets us to the combinatorial set-up, but we can also show inductively that each $L(\nu)$ satisfies the necessary even-odd vanishing condition to define an element of the ``enriched" Grothendieck group used in the proof: This is true for $\nu\in C^+$ by \cite[Thm. 3.12.1]{IKL}. If true for one $L(\nu)$ in an ascending sequence, it will be true for the next, call it $L(\nu')$, if the latter is a direct summand of the ``middle'' of a module obtained from a standard wall-crossing procedure. The latter is completely compatible with its analog for the larger group $G$,
 see \cite[Thm. 5.2(b)]{IKL}, but it is easier for it to be completely reducible for $G_1T$ than for $G$. If we assumed $p\geq 2h-3$, we could argue that $L(\nu')$ was the direct summand of the ``middle" (which would even be completely reducible) from validity of the LCF for weights of $\Gamma$ in the Jantzen region.  We could then complete the induction and claim \cite[Thm. 5.7]{IKL} held for $\Gamma$, and consequently equation (\ref{firstcongruence}) above (arguing further as in \cite[Thm. 5.8]{IKL}.
 
    However, we will assume only that $p>h$, and argue differently to obtain the same complete reducibility at the $G_1T$ level. We broaden the induction, making use of a consequence of (\ref{firstcongruence}) in this section, namely (\ref{surj}) below. Let $\Gamma_0$ be the set of all weights in $\Gamma$ whose lengths are at most that of $\nu$ in the previous paragraph. We can assume that \cite[Thm. 5.7]{IKL} holds for all restricted irreducible modules for highest weights in $\Gamma_0$. It follows that \cite[Thm. 5.8]{IKL} and (\ref{firstcongruence}) hold for all $p$-translates of restricted weights in $\Gamma_0$, and further consequences noted in this section, such as (\ref{surj}). In particular, we can equate $\Ext^1$-calculations for ${\mathcal C}_k$ and ${\mathcal C}_K$ between irreducible modules with such highest weights. The character of the ``middle" is the same for ${\mathcal C}_k$ as that for its ${\mathcal C}_K$ analog. Also, since the LCF is assumed for $\Gamma$, the characters of irreducible ${\mathcal C}_K$ modules appearing in the ``middle" all reduce ``mod $p$" to irreducible ${\mathcal C}_k$ modules (even that of $L(\nu')$). Now for any ${\mathcal C}_K$ irreducible module $L(\omega)$ appearing in the ``middle", its multiplicity can be determined as the dimension of $\Ext^1_{{\mathcal C}_K}(L(\nu),L(\omega))$, or of the same $\Ext^1$ group with its two arguments reversed. For $\omega\not=\nu'$, this $\Ext^1$ group has the same dimension, by application of (\ref{surj}) for $\Gamma_0$, as that for ${\mathcal C}_k$. Consequently,  all irreducible ${\mathcal C}_k$ composition factors $L(\omega)$, $\omega\not= \nu'$, of the ``middle" appear with their full multiplicity in both its head and socle. For $\omega = \nu'$ the multiplicity of $L(\nu')$ is 1. It follows that the ``middle" is completely reducible, and the induction is complete.

\begin{lem} \label{evenodd} (Even-odd vanishing) Assume that $p>h$ and let $\Gamma$ be a stable poset ideal of $p$-regular dominant weights. Assume that the Kazhdan-Lusztig property holds for all $\gamma\in\Gamma$; see (\ref{KLproperties}). Then for restricted weights $\lambda,\mu\in \Gamma$ and $n\in\mathbb Z$,
$$ \Ext^n_{G_1}(L(\lambda),L(\mu))\not=0\implies \Ext^{n\pm 1}_{G_1}(L(\lambda),L(\mu))=0.$$
\end{lem}

\begin{proof} If $\Ext^n_{G_1}(L(\lambda),L(\mu))\not=0$, then for some $\theta\in X(T)$,
$\Ext^n_{G_1T}(L(\lambda),L(\mu+p\theta))\not=0$. Thus, $\ell(\lambda)-\ell(\mu+p\theta)\equiv n$ mod 2, using
(\ref{firstcongruence}).  Also, $\lambda$ and $\mu+p\theta$
are $W_p$-conjugate, so that $\lambda-\mu-p\theta$ lies in the root lattice ${\mathbb Z}R$.
For the same reason, if $\Ext^{n+1}_{G_1}(L(\lambda),L(\mu))\not=0$, then for some $\theta'\in X(T)$,
$\Ext^{n+1}_{G_1T}(L(\lambda),L(\mu+p\theta'))\not=0$. This implies that $\ell(\lambda)-\ell(\mu+p\theta')
\equiv n+1$ mod 2. Again, $\lambda-\mu-p\theta'\in{\mathbb Z}R$. Therefore,
$p(\theta-\theta')\in{\mathbb Z}R$.  Since $p>h$, $p$ is relatively prime to the index of connection of
$R$, so that $\theta-\theta'\in {\mathbb Z}R$, and, therefore, by (\ref{length}) above, $\ell(\mu+p\theta)
\equiv \ell(\mu+p\theta')$ mod 2. Putting things together, we get that $n+1\equiv n$ mod 2, which is
absurd. Thus, $\Ext^{n+1}_{G_1}(L(\lambda),L(\mu))=0$. The same argument shows that
$\Ext^{n-1}_{G_1}(L(\lambda),L(\mu))=0.$
 \end{proof}

Let $A$ be a positively graded algebra. For a graded $A$-module $N$ and an integer $r$, let $N\langle r\rangle$ be the shifted graded $A$-module, obtained by putting $N\langle r\rangle_s:=N_{s-r}$. If $M,N$ are graded $A$-modules, let $\grExt^n_A(M,N)$ be the $n$th Ext-group computed in the category of graded
$A$-modules.

\begin{thm}\label{maintheorem} Assume that $p>h$. Let $\Gamma$ be a stable poset ideal in $X_{\text{\rm reg}}(T)_+$. Assume that if $\gamma\in\Gamma$ is $p$-restricted,
then the LCF holds for $L(\gamma)$. 
 If $\lambda,\mu\in\Gamma$ and $r\in\mathbb Z$, then
$$\grExt^n_{u'}(L(\lambda),L(\mu)\langle r \rangle )\not=0\implies n=r.$$
\end{thm}

\begin{proof} The hypothesis implies that $\rDelta(\gamma)\cong \rnabla(\gamma)=L(\gamma)$, for all $p$-restricted dominant
weights $\gamma\in\Gamma$ (or, more generally, for all $\gamma\in\Gamma$). Thus, if $\gamma\in\Gamma$
is restricted, $L(\gamma)\cong k\otimes_\sO \widetilde L_\zeta(\gamma)$ as discussed right before
the statement of the theorem.

Let $\lambda,\mu\in\Gamma$ be $p$-restricted. Form the short exact sequence 
$$0\to \wL_\zeta(\mu)
\overset\pi\longrightarrow \widetilde L_\zeta(\mu)\longrightarrow L(\mu)\to 0$$
of $\widetilde u'$-modules.
Write $L=L(\lambda), L'=L(\mu)$, etc. and form the long exact sequence of $\Ext$-groups
$$\cdots\to \Ext^{n-1}_{\widetilde u_\zeta}(\widetilde L,L')\to \Ext^n_{\widetilde u_\zeta}(\widetilde L,\widetilde L')
\overset\pi\to\Ext^n_{\widetilde u'}(\widetilde L,\widetilde L')\to\Ext^n_{\widetilde u'}(\widetilde L,L')\to \Ext^{n+1}_{\widetilde u'}(\widetilde L,\widetilde L')\to \cdots
$$
Observe that $\Ext^\bullet_{\widetilde u'}(\wL,L')\cong\Ext^\bullet_{u'}(L,L')$. 

Now assume that $\grExt^n_{u'}(L,L')\not=0$. Then, by Lemma \ref{evenodd}, $\Ext^{n+1}_{u'}(L,L')=0$ and
Nakayama's lemma, and the long exact sequence above force $\Ext^{n+1}_{\widetilde u}(\wL,\wL')=0$. 
Thus, 
\begin{equation}\label{surj}\Ext^n_{\widetilde u'}(\wL,\wL')/\pi\Ext^n_{\widetilde u'}(\wL,\wL')\cong\Ext^n_{u'}(L,L'),\quad\forall n\geq 0.\end{equation}
In addition, $\Ext^{n-1}_{u'}(\wL,L')\cong\Ext^{n-1}_{u'}(L,L')=0$, so that $\Ext^n_{\widetilde u'}(\wL,\wL')$
is free of rank equal to the dimension of $\Ext^n_{u'}(L,L')$ or of $\Ext^n_{\widetilde u'_K}(L_K,L'_K)$. In particular,
$\Ext^n_{\widetilde u'_K}(L_K,L'_K)\cong\Ext^n_{\widetilde u'}(\wL,\wL)_K$.

On the other hand, $u'$ inherits the structure of a positively graded algebra from the grading on
$\widetilde u'$. We have 
$$\begin{cases}\Ext^n_{\widetilde u'}(\wL,\wL')\cong\bigoplus_{r}\grExt^n_{\widetilde u'}(\wL,\wL'\langle r\rangle);\\
\Ext^n_{u'}(L,L')\cong\bigoplus_r\grExt^n_{u'}(L,L'\langle r\rangle).\end{cases}$$
It follows that the isomorphism (\ref{surj}) induces an isomorphism 
$$\grExt^n_{\widetilde u'}(\wL,\wL'\langle r\rangle )/\pi\, \grExt^n_{\widetilde u'}(\wL,\wL'\langle r\rangle )  \cong\grExt^n_{u'}(L,L'\langle r\rangle).$$ Hence, 
$\grExt^n_{\widetilde u'}(\wL,\wL'\langle r\rangle )\not=0$, as is $\grExt^n_{\widetilde u'_K}(\wL_K,\wL'_K\langle r\rangle )\not=0$. Therefore,
because $\widetilde u'_K$ is a Koszul algebra, $n=r$, completing the proof.
\end{proof}

\begin{rem}\label{remark5.3} Consider the modules $Z_K(\lambda)$ and $Z'_K(\lambda)$ in the quantum category ${\mathcal C}_K$
defined in \cite[\S2.11]{AJS}
$\lambda\in X$. By \cite[\S\S8.8--8.12]{AJS}, these modules (for $p$-regular weights $\lambda$) are
$\mathbb Z$-graded modules, denoted $\widetilde Z_K(\lambda)$ and $\widetilde Z'_K(\lambda)$, for the Koszul algebra $u'_\zeta$. In fact, by \cite[Prop. 18.19(b)]{AJS},
these graded modules are linear.  It can be shown that the modules $\widetilde Z_K(\lambda)$ and $\widetilde Z'_K(\lambda)$ admit graded $\sO$-forms $\widetilde Z_\sO(\lambda)$ and $\widetilde Z'_\sO(\lambda)$. Hence, base changing to
the field $k$, we see that the classical modules $Z_k(\lambda)$ and $Z_k'(\lambda)$ have induced
gradings. Then, it can be shown that, if $\lambda\in X(T)_+$ and $\mu\in\Gamma$, then 
$\grExt^n_{u'}(Z_k(\lambda), L(\mu)\langle r\rangle)\not=0$ implies $n=r$. A similar result holds for the
$Z'_k(\lambda)$.  \end{rem}

\section{Some relative results}We begin with the following general result. It does not require any
assumption of the LCF on $\Gamma$. We will work with the quasi-hereditary algebra $\wgr A_\Gamma$,
which has weight poset $\Gamma$, standard modules $\Delta_{\wgr A_\Gamma}(\gamma)$, and
costandard modules $\nabla_{\wgr A_\Gamma}(\gamma)$, $\gamma\in\Gamma$. Also,
$\Delta_{\wgr A_\Gamma}(\gamma)=\wgr\Delta(\gamma)$ and there is a dual construction (in the same
spirit) for
$\nabla_{\wgr A_\Gamma}(\gamma)$; see \cite[(4.0.2)]{PS11}, where $\nabla_{\wgr A_\Gamma}(\gamma)
=\wgr^\diamond \nabla(\gamma)$.

\begin{thm}\label{thirdmainconsequencethm6.5} \cite[Thm. 6.5]{PS13} Assume that $p\geq 2h-2$ is
an odd prime, and let $\Gamma$ be a finite poset ideal of $p$-regular weights.

 (a) For $\lambda,\mu\in\Gamma$ and any integer $n\geq 0$, there are natural vector space isomorphisms
\begin{equation}\label{equiv2}\begin{aligned}
\Ext^n_{\wgr A_\Gamma}(\Delta_{\wgr A_\Gamma}(\lambda),\rnabla(\mu))&\cong\Ext^n_{A_\Gamma}(\Delta(\lambda),\rnabla(\mu))\\
&\cong\Ext^n_G(\Delta(\lambda),\rnabla(\mu))\end{aligned}\end{equation}
and
\begin{equation}\label{equiv3}\begin{aligned}
\Ext^n_{\wgr A_\Gamma}(\rDelta(\lambda),\nabla_{\wgr A_\Gamma}(\mu))&\cong\Ext^n_{A_\Gamma}(\rDelta(\lambda),\nabla(\mu))\\
&\cong\Ext^n_G(\rDelta(\lambda),\nabla(\mu)).\end{aligned}\end{equation}

(b) 
 For any integer $n\geq 0$, there are natural vector space isomorphisms
\begin{equation}\label{equiv1}\begin{aligned}
\Ext^n_{\wgr A_\Gamma}(\rDelta(\lambda),\rnabla(\mu)& )\cong\Ext^n_{A_\Gamma}(\rDelta(\lambda),\rnabla(\mu))\\
&\cong\Ext^n_G(\rDelta(\lambda),\rnabla(\mu)).\end{aligned}\end{equation} 
for $\lambda,\mu\in\Gamma$.
\end{thm}

Now let $\Gamma$ be a finite stable poset ideal of $p$-regular weights. Let $a(\Gamma)$ be the
number of $p$-alcoves $C$ which intersect $\Gamma$ non-trivially.  By the argument for \cite[Prop. 10.3]{PS9}, $A_\Gamma$
has global dimension $\leq 2a(\Gamma)$.  

\begin{thm}\label{theorem6.2} Assume that the LCF holds on $\Gamma$. Assume that $p>6a(\Gamma)+3h-4$. For $\gamma,\nu\in\Gamma$, $n\in\mathbb N$,
$m\in\mathbb Z$, we have
$$\begin{cases}(1) \,\,\grExt^n_{\wgr A_\Gamma}(\rDelta(\gamma),\rnabla(\nu)\langle m\rangle )\not=0\implies n=m;\\
(2)\,\,\grExt^n_{\wgr A_\Gamma}(\wgr\Delta(\gamma),\rnabla(\nu)\langle m\rangle)\not=0\implies n=m.
\\
(3) \,\,\grExt^n_{\wgr A_\Gamma}(\rDelta(\gamma),\nabla_{\wgr A_\Gamma}(\nu)\langle m\rangle)\not=0\implies n=m.
\end{cases}$$
\end{thm}  

\begin{proof} We first prove (1). Assume that $6a(\Gamma)+3h-4$. Write $A=A_\Gamma$. By Theorem \ref{thirdmainconsequencethm6.5}, 
 \begin{equation}\label{global}\Ext^n_{\wgr A}(\rDelta(\gamma),\rnabla(\nu))\cong\Ext^n_A(\rDelta(\gamma),\rnabla(\nu)).\end{equation}
 As noted above, $A$ has global dimension at most $2a(\Gamma)$. Thus, the terms in (\ref{global}) vanish
 if $n> 2a(\Gamma)$. 
 
 On the other hand,   there is a Hochschild-Serre spectral sequence
 \begin{equation}\label{spectralsequence}
 E_1^{s,t}=H^s(G, \Ext^t_{G_1}(\rDelta(\gamma),\rnabla(\nu)^{[-1]})\Longrightarrow
 \Ext^{s+t}_G(\rDelta(\gamma),\rnabla(\nu)).\end{equation}
 We call a weight $\lambda\in X(T)$, $b$-small provided that $|(\lambda,\alpha^\vee)|\leq b$ for all
 positive roots $\alpha$. If $M$ is a finite dimensional rational $G$-module, then it is called $b$-small
 provided that its weights are all $b$-small. Write $\gamma=\gamma_0+p\gamma_1$ and $\nu=\nu_0+p\nu_1$, with $\gamma_0,\nu_0\in X_1(T)$ and $\gamma_1,\nu_1\in X(T)_+$. Then 
$$ \Ext^t_{G_1}(\rDelta(\gamma),\rnabla(\nu))^{[-1]}\cong\Ext^t_{G_1}(L(\gamma_0),L(\nu_0))^{[-1]}\otimes\nabla(\gamma_1^{\star})\otimes\nabla(\nu_1).$$
Here $\gamma_1^\star=-w_0\gamma_1$ is the image of $\gamma_1$ under the opposition involution.
Now \cite[Cor. 3.6]{PSS} implies that, if $t$ is any non-negative integer,  the modules $\Ext^t_{G_1}(L(\gamma_0),L(\nu_0))^{[-1]}$ are
$(3t+2h-3)$-small.  Therefore, if $t\leq 2a(\Gamma)$ and $p>6a(\Gamma)+3h-4$, the dominant weights $\xi$ in $\Ext^t_{G_1}(L(\gamma_0), L(\nu_0))^{[-1]}$ lie in the bottom $p$-alcove $C^+$. Hence, in this case, $\Ext^t_{G_1}(L(\gamma_0),L(\nu_0))^{[-1]}$ has a $\nabla$-filtration. Thus, $\Ext^t_{G_1}(\rDelta(\lambda),\rnabla(\nu))^{[-1]}$ has a $\nabla$-filtration. Therefore, as long as $n:=s+t\leq 2a(\Gamma)$, the hypotheses guarantee
that the spectral sequence (\ref{spectralsequence}) collapses giving that
 \begin{equation}\label{collapse}\Ext^n_G(\rDelta(\gamma),\rnabla(\nu))\cong\Ext^n_{G_1}(\rDelta(\gamma),\rnabla(\nu))^G.\end{equation}
 This is because, if $M$ is a rational $G$-module with a $\nabla$-filtration, then $H^m(G,M)=0$ for all
 $m>0$.
 On the other hand, if $n\geq 2a(\Gamma)$, then, as pointed out in the previous paragraph, we
 have $\Ext^n_G(\rDelta(\gamma),\rnabla(\nu))=0$. 
 
 Thus, if $\Ext^n_{\wgr A}(\rDelta(\gamma),\rnabla(\nu))\not=0$, we obtain that 
 $$\Ext^n_{\wgr A}(\rDelta(\gamma),\rnabla(\nu))\cong\Ext^n_A(\rDelta(\gamma),\rnabla(\nu))$$ injects into $\Ext^n_u(\rDelta(\gamma),
 \rnabla(\nu))$. Returning to the level of graded modules, it follows easily that 
 $$\grExt^n_{\wgr A}(\rDelta(\gamma),\rnabla(\nu)\langle r\rangle)\subseteq\grExt^n_{u'}(\rDelta(\gamma),
 \nabla(\nu)\langle r\rangle), \quad\forall n.$$ 
 Then (1) follows from Theorem \ref{maintheorem}.

Finally, we sketch the proof of (2), leaving the dual proof of case (3) to the reader. Because of the condition imposed on $p$, Theorem \ref{thirdmainconsequencethm6.5} implies that we can assume that $n\leq 2a(\Gamma)$ and
then $\Ext^n_{G_1}(\rDelta(\nu),\rnabla(\mu))^{[-1]}$ has a $\nabla$-filtration.  In addition, the rational
$G$-module $\Ext^n_{G_1}(\Delta(\nu),\rnabla(\mu))^{[-1]}$ has a 
$\nabla$-filtration. With this condition, the reader may check that the proof of Theorem \ref{filtration6.2} (namely, \cite[Thm. \ref{cor4.2}]{PS13}) remains valid. Hence, the natural map (\ref{Jantzentheorem}) is surjective. Therefore, passing to $\grExt_{u'}^n$, we obtain that, if $\grExt^n_{\wgr A_\Gamma}(\wgr\Delta(\gamma),\rnabla(\nu)\langle m\rangle)\not=0$, then $\grExt^n_{u'}(\rDelta(\gamma),\rnabla(\nu)\langle m\rangle)\not=0$. Hence, 
as before, $n=m$.
 \end{proof}

\noindent{\it Scholia}.   In this section and the previous one we have proved results which
enable the ``relativization" of hypotheses on the underlying regular
dominant weight posets $\Gamma$ of most results in \S\S3,4. By
``relativization" we mean replacing any hypothesis that ``The Lusztig
character \ref{cor4.2}ula holds" with $\Gamma$ is stable and the Lusztig character formula holds for
$\Gamma$ (see (\ref{LCF}) for terminology).   We explain how this works for each of the
results in \S\S3,4:

Theorem \ref{Weyl} is already relativized in \cite[Thm. 7.1]{PS13}. In fact, the
formulation does not even require that  that $\Gamma$ be stable; and the LCF is
effectively required only on the poset of non-maximal elements of $\Gamma$. Here it is required that 
$p\geq 2h-2$ is odd.
Theorem \ref{theorem5.3} is not yet fully relativized. Nevertheless, its conclusions hold
without any LCF hypothesis, if $p$ is sufficiently large relative the Ext
degree. This $\Ext^n$ group (in the statement of the theorem)  is generally
interesting only when the cohomological degree $n\leq {\text{\rm gl.dim.}}\,A_\Gamma$. See the proof of Theorem \ref{theorem6.2}. If we assume the LCF
holds for $\Gamma$, this global dimension is at most $2a(\Gamma)$, where $a(\Gamma)$
is the number of alcoves intersecting $\Gamma$ nontrivially. For cohomological degree $n$ at
most this value, the conclusion of the Theroem \ref{theorem5.3} holds if $p>6a(\Gamma)+3h-4$.
We regard such primes $p$ as ``fairly large", but not huge.
Theorem \ref{firstmainconsequencethm5.6} is relativized for $p$ fairly large. Explicitly, $p>6a(\Gamma)+3h-4$
as above. This is established in Theorem \ref{theorem6.2}.
Theorem \ref{filtration6.2} is not yet fully relativized. Nevertheless,  its  $\nabla$-filtration
conclusions hold for $\Ext^n$ with $n\leq 2a(\Gamma)$, provided  $p$ is fairly
large, as above.  Again, if the LCF holds for  $\Gamma$, this includes all cohomological
degrees $n \leq {\text{\rm gl. dim.}} A_\Gamma$. The surjectivity
assertions of the theorem can be proved, under these assumptions on $\Gamma$
and $p$, as in the argument for Theorem \ref{theorem6.2} in \cite{PS13}. 
Theorem \ref{secondmainconsequencethm3.7} is relativized, as proved in Theorem \ref{theorem6.2} in this paper.
Corollary \ref{corollary3.8}, as a special case of Theorem \ref{secondmainconsequencethm3.7}, is also relativized. 
Corollaries \ref{cor4.2} and \ref{cor4.3} are relativized, using Theorem \ref{secondmainconsequencethm3.7} and Corollary
\ref{corollary3.8}
discussed above.
Similarly Proposition \ref{stlin} is relativized. That is, in these results in
\S4, the assumption that the LCF holds may be replaced with the
assumption that $\Gamma$ is stable and the LCF holds on $\Gamma$, as discussed
above. 
All other results in \S4 are stated in an abstract context, and so have
no need of relativization. This concludes our discussion.

\section{Open questions}Many of the results in our ``forced grading program" go back to the paper \cite{PS9},
entitled ``A new approach to the Koszul property using graded subalgebras."  One of the subalgebras in the title is the $p$-regular part $u'_\zeta$ of the small quantum group (associated to $G$).  If $p>h$,  it may
be deduced from \cite[\S\S\S17--18]{AJS} that $u'_\zeta$ is a Koszul algebra.  This fact needed for this deduction is the
validity of the LCF for $u'_\zeta$ when $p>h$.  This is now known for $p>h$ \cite{T},
and the LCF holds also for $u_\zeta$.  Nevertheless, the corresponding Koszulity of $u_\zeta$ is not
known.  

\begin{ques}\label{question1} Assume $p>h$. Is it true that the small quantum enveloping algebra $u_\zeta$ at a
$p$th root of unity is a Koszul algebra?\end{ques}

Curiously,
such an extension of the work of \cite{AJS} already exists in positive characteristic \cite{Riche} (for ``really"
large $p$). But there is no extension 
so far in the (presumably easier) quantum case. 

One can also hope that a similar result holds at
the integral level. 

\begin{ques}\label{question2} Assuming the answer to Question \ref{question1} is positive, does the graded algebra
$u_\zeta$ admit a compatible $\sO$-form as in Theorem \ref{thm2.2} (or in \cite[\S8]{PS11})? \end{ques}

It seems likely that, at least in classical types, these questions might also have positive answers for
some small prime cases, i.~e., $p\leq h$. Positive answers to these questions would likely lead to
obtaining graded, integral quasi-hereditary algebras as in \cite{PS10}, and it could also lead to the results in
\cite{PS13}, as well as to showing the result in \S4 above are  valid in the singular case. 

Finally, we raise

\begin{ques} Do the modules $Z_K(\lambda)$ discussed in
Remark \ref{remark5.3} satisfy the linearity (or Koszul) condition when $\lambda$ is allowed to be
singular? A dual property for the $Z'_K(\lambda)$ should also hold. (This question also appears to be open for the algebraic group schemes $G_1T$ in positive
characteristic, even for very large $p$.) \end{ques}

\section{Appendix: Comparison of gradings}The aim of this appendix is to flesh out part of the argument
for \cite[Prop. 18.21]{AJS} dealing with compatibilities of some $\mathbb Z$-gradings discussed there with 
natural weight gradings. We give two arguments. First, our own, is given after a general discussion of the
context and issues, and deals fairly directly with the weight gradings involved. The second argument, which we
give 
briefly, is our interpretation of the argument sketched in \cite[18.21]{AJS} itself.

 Let $Y$ be an abelian group. The appendix \cite[Appendix E]{AJS} defines 
a $Y$-category to be an additive category $\mathcal C$ equipped with shift functors $\langle \xi\rangle:{\mathcal C}\to{\mathcal C}$, one for
each $\xi\in Y$. Certain natural conditions must be satisfied. In a classical sense, there is also the notion of a $Y$-graded
algebra $A$ and $Y$-graded modules for it, which \cite[\S E.3]{AJS} notes gives rise to a $Y$-category (with obvious shift functors). Moreover,
\cite[E.3]{AJS} defines the notion of a $Y$-generator for an abelian $Y$-category $\mathcal C$ and shows, for
any projective $Y$-generator $P$ that $\End_{\mathcal C}^\sharp(P)$ is a $Y$-graded algebra
(and $E:=\End_{\mathcal C}^\sharp(P)^\op$ is $Y$-graded in the same way).\footnote{ Here
$\Hom_{\mathcal C}^\sharp(M,N):=\bigoplus_{\xi\in Y}\Hom_{\mathcal C}(M,N\langle \xi\rangle)$,  for any pair
 of objects $M,N\in{\mathcal C}$ .   } The original category
$\mathcal C$ is equivalent to the module category $E$-grmod of $Y$-graded $E$-modules. In case $\mathcal C$ is the category of $Y$-graded
modules for a $Y$-graded algebra $A$, then $\Hom_{\mathcal C}^\sharp(M,N)\cong\Hom_A(M,N)$ (ungraded
$\Hom_A$). 

Two important examples, introduced early in \cite{AJS},  of categories graded by an abelian group are the categories
we call $\ssC_k$ (namely, the category of finite dimensional rational $G_1T$-modules) and its quantum analogue, which we call $\ssC_K$; see \cite[\S2.4]{AJS}.\footnote{ Both these categories are denoted
${\mathcal C}_k$ in \cite{AJS}, distinguished only by a ``Case 1'' ($G_1T$ case) or ``Case 2" (quantum analogue case) context.} The abelian group generally used is $Y:=p{\mathbb Z}R$, especially in studying blocks, though
the larger group $pX$ can sometimes be used, where $X=X(T)$
as in \S1 of this paper.  If $M$ is any object in ${\mathcal C}_k$, then $M\otimes p\xi$ is also
in ${\mathcal C}_k$. Setting $M\langle p\xi\rangle:=M\otimes p\xi$ gives ${\mathcal C}_k$ the
structure of a $pX$-category.  These functors do not generally preserve blocks unless $p\xi\in Y$. However, any
block is a $Y$-category.
 
 Of course, $M$ has a classical weight space decomposition---namely, $M=\bigoplus_{\omega\in X}M_\omega$. However, this decomposition of $M$ does not correspond to the weight space decomposition, using a
 graded endomorphism algebra as in the first paragraph above:   Let $P\in{\mathcal C}_k$ be such that
 $P|_u\cong u$ (as left $u$-modules) in Case 1 and $P|_{u_\zeta}\cong u_\zeta$ in Case 2.  For instance,
take $P:=\bigoplus_{\lambda\in X_1(T)} \Phi_k(\lambda)$, where the construction $\Phi_k(-)$ is described in  \cite[2.6 (3)]{AJS}. In Case 1, $\Phi_k(\lambda):=(\ind_T^{G_1T}-\lambda)^*$ is the module obtained by coinducing from $T$ to $G_1T$  the one-dimensional $T$-module defined by $\lambda$, and there is an analogous
 construction in Case 2.  (Here $X_1(T)$ could be replaced by any collection $X_1'$ of coset representatives of $pX$ in $X$.)
We have
 $\Hom_{{\mathcal C}_k}(\Phi_k(\lambda), M)\cong M_{\lambda}$, $\lambda\in X_1(T)$. 
 The module $P$ is a projective $pX$-generator for ${\mathcal C}_k$ in the sense of paragraph 1 above, giving an equivalence of ${\mathcal C}_k$ with the category of $pX$-graded modules for a $pX$-graded
 algebra $E=\End^\sharp_{{\mathcal C}_k}(P)^\op$. The equivalence is given by $M\mapsto \Hom_{{\mathcal C}_k}^\sharp(P,M)$. The latter module is isomorphic to $M$ as a vector space and has decomposition
 $M=\bigoplus_{\theta\in X}M(p\theta)$ into $pX$-grades. We find that, in terms of the original weight
 space decomposition of $M$, that 
   $M(p\theta)=\bigoplus _{\lambda\in X_1(T)}M_{\lambda+p\theta}$. Notice that the $X$-weight spaces do {\it not} correspond exactly to the $pX$-weight
spaces (which are made up of a sum of many of the former weight spaces).  This suggests we cannot just use the results
on graded categories in \cite[Appendix E]{AJS} to obtain $X$-weight space compatibility with the $\mathbb Z$-gradings in 
\cite[\S\S17,18]{AJS}.\footnote{One comes closer by thinking, in Case 1, of fully embedding the category of 
$G_1T$-modules into the category of rational modules for the semidirect product $G_1.T:=G_1\rtimes T$ 
(using the group scheme surjection $G_1. T\twoheadrightarrow G_1T$). Modules for $G_1. T$ are just
$X$-graded modules for the $X$-graded algebra $u$. In this category, it is possible to make sense of a
tensor product $M\otimes\xi$ for any $\xi\in X$, thus fully embedding the $Y$-category ${\mathcal C}_k$
into an $X$-category in a natural way. See \cite[18.20]{AJS} which takes a similar approach for blocks.}

However, the compatibility exists just the same, at least for modules $M$ whose composition factors have $p$-regular highest weights. Before discussing this, we give more details on how the above formalism works for blocks. Let
${\mathcal C}_k(\Omega)$ be the block of ${\mathcal C}_k$ associated to a $W_p$-orbit $\Omega$ of $p$-regular
weights. (We will not use the $p$-regularity in this paragraph.) All the weights of any $M$ in ${\mathcal C}_k(\Omega)$, $\Omega=W_p\cdot\omega$, belong to a single coset of $\omega+{\mathbb Z}R$ of the root lattice ${\mathbb Z}R$. Working with any prime $p$ relatively prime to $[X:{\mathbb Z}R]$ (such as any prime
$p>h$), choose a set of coset representatives $X_1'$ for $pX$ in X such that $X_1'\subseteq \omega+{\mathbb Z}R$. Construct $P'=\bigoplus_{\lambda\in X'_1} \Phi_k(\lambda)$, analogous to the module $P$
above. Let $P'_\Omega$ be the projection of $P'$ onto the block ${\mathcal C}_k(\Omega)$. Then $P'_\Omega$ is a $Y$-generator for the $Y$-category ${\mathcal C}_k(\Omega)$. Also, $P'_\Omega=
\bigoplus_{\lambda\in X'_1}\Phi_\Omega(\lambda)$, where $\Phi_\Omega(\lambda)$ is the projection of $\Phi_k(\lambda)$ onto ${\mathcal C}_k(\Omega)$. Let  $Q_k(\nu)$, be the projective cover of the irreducible module 
$L_k(\nu)$, $\nu\in\Omega$; see  \cite[\S4.15]{AJS}. The module $\Phi_\Omega(\lambda)$ is just the direct sum of copies of the $Q_k(\nu)$, $\nu\in\Omega$, each appearing with the same multiplicity (possibly zero) that $Q_k(\nu)$ appears as a summand of $\Phi_k(\lambda)$.  For any module $M$ in ${\mathcal C}_k(\Omega)$ and
any $\lambda\in X'_1$, we 
have
\begin{equation}\label{formula}
\Hom_{{\mathcal C}_k(\Omega)}(\Phi_\Omega(\lambda),M)\cong\Hom_{{\mathcal C}_k}(\Phi_k(\lambda),M)\cong M_\lambda.\end{equation}
We can recover any weight space $M_\nu$ of $M$ by writing $\nu=\lambda+p\theta$, for some
$\lambda\in X_1'$ and some $\theta\in{\mathbb Z}R$. (The fact that this can be done depends heavily
on the construction of $X'_1$
described above: choose $\lambda\in X'_1$ so that $\nu\in\lambda+ pX$. Since $X'_1$ and $\nu$ belong
to the same coset of ${\mathbb Z}R$, we have $\lambda-\nu\in pX\cap {\mathbb Z}R=p{\mathbb Z}R$ 
here.) Then
\begin{equation}\label{formula2}
\Hom_{{\mathcal C}_k(\Omega)}(\Phi_\Omega(\lambda),M\langle -p\theta\rangle)\cong M_{\lambda+p\theta}=M_\nu.\end{equation}
That is, we have completely recovered the $X$-grading of $M$ using the $Y$-category ${\mathcal C}_k(\Omega)$.  

Next, consider the issue of passing from $Y$-compatibility of a $\mathbb Z$-graded version of ${\mathcal C}_k(\Omega)$ to $X$-compatibility. \cite[\S18]{AJS} introduced various 
objects in ${\mathcal C}_k(\Omega)$ which have ``graded forms," putting them in an 
associated $Y\times{\mathbb Z}$-category $\widetilde{\mathcal C}_k(\Omega)$. Among these
are the projective covers $Q_k(\lambda)$ of the irreducible module 
$L_k(\lambda)$, $\lambda\in\Omega$; see  \cite[\S18.16]{AJS}. This gives a $Y\times\mathbb Z$-graded 
version $\widetilde Q_k(\lambda)$ in $\widetilde{\mathcal C}_k(\Omega)$ of $Q_k(\lambda)$.  For $\lambda\in X'_1$, define $\widetilde\Phi_\Omega(\lambda)
=\bigoplus_{\nu\in\Omega} \widetilde Q_k(\nu)^{\oplus[\Phi_k(\lambda):Q_k(\nu)]}$. Thus, $\Phi_\Omega(\lambda)$ is obtained from $\widetilde\Phi_\Omega(\lambda)$ in $\widetilde{\mathcal C_k}(\Omega)$
by forgetting the $\mathbb Z$-grading.  We still have the isomorphism (\ref{formula}) and (\ref{formula2}),
when $\Phi_\Omega(\lambda)$ is replaced with any isomorphism copy, such as that obtained by forgetting
the $\mathbb Z$-grading on $\widetilde\Phi_\Omega(\lambda)$. 

Now let $M$ in (\ref{formula}) and (\ref{formula2}) be any object in ${\mathcal C}_k(\Omega)$ obtained
by forgetting the $\mathbb Z$-grading on an object in $\widetilde M\in\widetilde{\mathcal C}_k(\Omega)$. Then
the $\nu$-weight space of $M$ decomposes as a direct sum 
$$\bigoplus_{n\in\mathbb Z}\Hom_{\widetilde{\mathcal C}_k(\Omega)}(\widetilde\Phi_\Omega(\lambda),
\widetilde M\langle (-p\theta,-n)\rangle).$$
This establishes the needed $X$-compatibility from the $Y$-compatibility.

This passage from $Y$-compatibility to $X$-compatibility may be used in the discussion immediately
above \cite[Prop. 18.21]{AJS} to complete the proof of that proposition and its quantum analogue.  Essentially,
we agree that the algebra $\End^\sharp_{{\mathcal C}}(P)^\op$ discussed there is both $Y\times \mathbb Z$-graded
and $X$-graded. However, it seems to us that additional detail is required to make it $X\times\mathbb Z$-graded, and a version of this has been supplied above. A second way to do this, perhaps implicit in \cite{AJS}, is to use the direct sum decomposition \cite[18.20]{AJS} to define an $X\times\mathbb Z$-categroy, using shifts provided between summands. This process gives an $X\times\mathbb Z$-grading on $\End^\sharp_{\mathbb Z}(P)^\op$, and also leads to the desired $X$-compability.

\end{document}